\documentclass[12pt]{amsart}
\usepackage{amssymb,latexsym}
\usepackage{amsfonts}
\usepackage{amsmath}
\usepackage[colorlinks,linkcolor=blue,anchorcolor=blue,citecolor=blue]{hyperref}
\usepackage{algorithm}
\usepackage{enumerate}
\usepackage{algpseudocode}
\usepackage{verbatim}
\usepackage{graphicx}

\newcommand\F{{\mathbb F}}
\newcommand\Q{{\mathbb Q}}

\newcommand\Z{{\mathbb Z}}

\newtheorem{theorem}{Theorem}[section]
\newtheorem{lemma}[theorem]{Lemma}
\newtheorem{corollary}[theorem]{Corollary}

\theoremstyle{definition}

\newtheorem{example}[theorem]{Example}

\theoremstyle{remark}

\numberwithin{equation}{section}

\begin{document}

\title[Exceptional units]{On the additive and multiplicative structures of the exceptional units in finite commutative rings}

\author{Su Hu}
\address{Department of Mathematics, South China University of Technology, Guangzhou 510640, China}
\email{mahusu@scut.edu.cn}

\author{Min Sha}
\address{Department of Computing, Macquarie University, Sydney, NSW 2109, Australia}
\email{shamin2010@gmail.com}

\subjclass[2010]{11B13, 11D45, 11T24}

\keywords{Exceptional unit, finite commutative ring, character sum}

\begin{abstract} 
Let $R$ be a commutative ring with identity. A unit $u$ of $R$ is called exceptional if $1-u$ is also a unit. 
When $R$ is a finite commutative ring, we determine the additive and multiplicative structures of its exceptional units; 
and then as an application we find a necessary and sufficient condition under which $R$ is generated by its exceptional units. 
\end{abstract}

\maketitle

\section{Introduction}
\label{Intro}

\subsection{Background}

Let $R$ be a commutative ring with $1\in R$ and $R^{*}$ its group of units. A unit $u\in R$ is called \textit{exceptional} if $1-u\in R^{*}$. 
We denote by $R^{**}$ the set of exceptional units of $R$. 
This concept was introduced by Nagell~\cite{Nagell} in 1969 in order to solve certain cubic Diophantine equations. 
The key idea is that the solution of many Diophantine equations can be reduced to the solution 
of a finite number of unit equations of type 
$$
ax + by = 1, 
$$
where $x$ and $y$ are restricted to units in the ring of integers of some number field (see \cite{EG} for a treatise on unit equations).  
By choosing $a=b=1$, we obtain the concept of exceptional unit. 

Let $\mathbb{Z}_{n}$ be the residue class ring of the integers $\Z$ modulo a positive integer $n$. 
Sander~\cite[Theorem 1.1]{Sander2}  determined the number of representations of an element in $\mathbb{Z}_{n}$ as the sum of two exceptional units. 
Recently, Zhang and Ji \cite[Theorem 1.5]{Zhang} extended Sander's result to the case when $R$ is a residue class ring of a number field. 
Using a different approach with the aid of exponential sums, for any integer $k\ge 2$ 
Yang and Zhao \cite[Theorem 1]{Yang} obtained an exact formula for the number of ways to represent an element of $\Z_n$ as a sum of $k$ exceptional units. 
Most recently, Miguel \cite[Theorem 1]{Miguel} generalized the result of Yang and Zhao to the case of finite commutative rings. 
Sander \cite{Sander3} persued another structural perspective by investigating the atom decomposition of sets of 
exceptional units in $\Z_n$.

\subsection{Our situation}

In this paper, we consider the additive and multiplicative structures of the exceptional units of a finite commutative ring. 
We show that as in \cite{Miguel,Yang}, there is an exact formula for  the number of ways to represent a unit as a product of $k$ exceptional units. 
As an application and in combination with Miguel's result, 
we completely determine  the additive and multiplicative structures of such exceptional units in our situation 
(see Theorems \ref{thm:eu-sum} and \ref{thm:eu-prod} below); 
and then as an application we find a necessary and sufficient condition 
under which $R$ is generated by its exceptional units (see Corollaries \ref{cor:eu-sum} and \ref{cor:exunit2} below). 

From now on, $R$ is a finite commutative ring with identity.   
It is well-known that $R$ can be uniquely expressed as a direct sum of local rings; see \cite[page 95]{Mc}. 
So, in the following we assume that 
$$
R = R_1 \oplus \cdots \oplus R_n, 
$$
where each $R_i$, for $i=1,\ldots, n$, is a local ring. 
Then, each element $c \in R$ can be represented as $(c_1, \ldots, c_n)$ with $c_i \in R_i, i=1, \ldots, n$. 
For each $i=1,\ldots, n$, suppose that  $M_i$  is the unique maximal ideal of $R_i$, and  put 
$$
m_i = |M_i| , \qquad q_i = | R_i/M_i|. 
$$ 
Here, each residue field $R_i/M_i$ is a finite field, and thus $q_i$ is a power of a prime. 

We first state our results in Sections \ref{sec:unit}, \ref{sec:eu1}, \ref{sec:eu2} and \ref{sec:app}, and then we give proofs in Section \ref{Prof}.

\subsection{Additive structure of units}   
\label{sec:unit}

Before further work, we first determine the additive structure of $R^*$. 
For $c\in R$, we define the set 
\begin{equation*}
\Psi_{k,R}(c)=\{(x_{1},\ldots,x_{k})\in (R^{*})^{k} :x_{1}+\cdots+x_{k} = c\},
\end{equation*} 
and put 
\begin{equation*}\label{phi}
\psi_{k,R}(c)=|\Psi_{k,R}(c)|.
\end{equation*} 
That is, $\psi_{k,R}(c)$ is the number of ways to represent $c$ as a sum of $k$ units. 
Kiani and Mollahajiaghaei obtained the following formula for $\psi_{k,R}(c)$ (see \cite[Theorem 2.5]{KM}), 
which generalizes the result in \cite{Sander1}.

\begin{theorem}[\cite{KM}] \label{thm:psi} 
For  any integer $k\ge 2$ and any $c=(c_1,\ldots,c_n)\in R$, we have 
$$
\psi_{k,R}(c) = \prod_{i=1}^{n}   m_i^{k-1}q_i^{-1} \mu_{k,R_i}(c_i), 
$$
where 
$$
\mu_{k,R_i}(c_i)= \begin{cases}
(q_i-1)^{k} + (-1)^{k}(q_i-1)&  \textrm{if $c_i \in M_i$,} \\
(q_i-1)^{k} + (-1)^{k+1}& \textrm{if $c_i \in R_i \setminus M_i$.} 
\end{cases}
$$
\end{theorem}

Theorem \ref{thm:psi} can be directly used to establish the additive structure of the units of $R$.  

\begin{corollary} \label{cor:u-sum}
The following hold: 

{\rm (i)} If $q_i>2$ for each $i=1, \ldots, n$, then we have 
$$
R^{*} + R^{*} = R.
$$

{\rm (ii)} If $q_1=\cdots =q_s=2$ $(s\ge 1)$ and $q_j>2$ for each $j>s$, 
then for any integer $k\ge 2$ we have  
$$
\sum_{i=1}^{k} R^{*} = 
\begin{cases}
(\oplus_{i=1}^{s} M_i ) \oplus  (\oplus_{j>s} R_j) & \textrm{if $k$ is even,} \\ 
(\oplus_{i=1}^{s} R_i \setminus M_i ) \oplus  (\oplus_{j>s} R_j) & \textrm{otherwise.}
\end{cases}
$$
In particular, each element of $R$ is a sum of units if and only if $s=1$. 
\end{corollary}

In Corollary \ref{cor:u-sum} (ii), if $s=n$, then the part  $\oplus_{j>s} R_j$ does not exist. 

\begin{corollary}  \label{cor:unit1}
The ring $R$ is generated by its units  if and only if there is at most one $q_i$ equal to $2$. 
\end{corollary}

\subsection{Additive structure of exceptional units} 
\label{sec:eu1}

We now turn our attention to exceptional units. 
In Theorem \ref{thm:psi} , if we choose $k=2$ and $c=1$, then we directly get the size of $R^{**}$.

\begin{corollary}  \label{cor:eu-size}
We have 
$$
| R^{**}| = \prod_{i=1}^{n}   m_i (q_i-2). 
$$
\end{corollary}

We directly have: 

\begin{corollary}  \label{cor:exist1}
$R^{**} \ne \emptyset$ if and only if each $q_i \ge 3$, $i=1,\ldots, n$.  
\end{corollary}

For $c\in R$, define the set 
\begin{equation*}
\Phi_{k,R}(c)=\{(x_{1},\ldots,x_{k})\in (R^{**})^{k} :x_{1}+\cdots+x_{k} = c \},
\end{equation*} 
and denote 
\begin{equation*}\label{phi}
\varphi_{k,R}(c)=|\Phi_{k,R}(c)|.
\end{equation*} 
That is, $\varphi_{k,R}(c)$ is the number of ways to represent $c$ as a sum of $k$ exceptional units. 
Miguel gave an exact formula for $\varphi_{k,R}(c)$; see \cite[Theorem 1]{Miguel}.

\begin{theorem}[\cite{Miguel}]\label{thm:phi} 
For  any integer $k\ge 2$ and any $c=(c_1,\ldots,c_n)\in R$, we have 
$$
\varphi_{k,R}(c) =  \prod_{i=1}^{n} (-1)^k m_i^{k-1}q_i^{-1} \rho_{k,R_i}(c_i),
$$
where 
$$
 \rho_{k,R_i}(c_i) = q_i \sum_{\substack{j=0 \\ j \equiv c_i \,\textrm{\rm (mod $M_i$)}}}^{k}\binom{k}{j} +(2-q_i)^{k}-2^k.
$$
\end{theorem}

Using Theorem \ref{thm:phi}, we completely determine  the additive structure of $R^{**}$ in the following theorem. 
Note that by Corollary \ref{cor:exist1} we need to exclude the case when there is some $q_i$ equal to $2$. 

\begin{theorem} \label{thm:eu-sum}
Assume that each $q_i \ge 3$, $i=1,\ldots, n$. Then, the following hold: 

{\rm (i)} If each $q_i$ is greater than $4$, $i=1, \ldots, n$, we have 
$$
R^{**} + R^{**} = R. 
$$

{\rm (ii)} If $q_1=\cdots =q_s=3$ $(s\ge 1)$ and $q_j > 4$ for each $j>s$,  
then for any $k\ge 2$ we have 
$$
\sum_{i=1}^{k} R^{**} = 
\begin{cases}
(\oplus_{i=1}^{s} M_i)  \oplus (\oplus_{j>s} R_j) & \textrm{if $k \equiv 0 $ {\rm (mod $3$)},} \\
(\oplus_{i=1}^{s} 2+M_i)  \oplus (\oplus_{j>s} R_j) & \textrm{if $k \equiv 1$ {\rm (mod $3$)},} \\
(\oplus_{i=1}^{s} 1+M_i)  \oplus (\oplus_{j>s} R_j) & \textrm{if $k \equiv 2$ {\rm (mod $3$)}.}
\end{cases}
$$

{\rm (iii)} If $q_1=\cdots =q_t=4$ $(t \ge 1)$ and $q_j > 4$  for each $j>t$,  
then for any $k\ge 2$ we have 
$$
\sum_{i=1}^{k} R^{**} = 
\begin{cases}
(\oplus_{i=1}^{t}(M_i \cup 1+M_i)) \oplus (\oplus_{j>t} R_j)  & \textrm{if $k$ is even,} \\
(\oplus_{i=1}^{t}(R_i \setminus (M_i \cup 1+M_i))) \oplus (\oplus_{j>t} R_j)  & \textrm{otherwise.}
\end{cases}
$$

{\rm (iv)} Assume that $q_1= \cdots =q_s=3$ $(s\ge 1)$,  $q_{s+1}= \cdots =q_{s+t}=4$ $(t \ge 1)$, 
and $q_j > 4$  for each $j>s+t$. 
Then, for any $k\ge 2$ we have 
\begin{align*}
 \sum_{i=1}^{k} R^{**} 
 = \begin{cases}
(\oplus_{i=1}^{s} M_i)  \oplus (\oplus_{i=s+1}^{s+t}(M_i \cup 1+M_i)) \oplus (\oplus_{j>s+t} R_j) \\
   \qquad\qquad\qquad\qquad\qquad\qquad\qquad\qquad  \textrm{if $k \equiv 0 $ {\rm (mod $6$)},} \\
(\oplus_{i=1}^{s} 2+M_i)  \oplus (\oplus_{i=s+1}^{s+t}(R_i \setminus (M_i \cup 1+M_i))) \oplus (\oplus_{j>s+t} R_j) \\
\qquad\qquad\qquad\qquad\qquad\qquad\qquad\qquad \textrm{if $k \equiv 1$ {\rm (mod $6$)},} \\
(\oplus_{i=1}^{s} 1+M_i)  \oplus (\oplus_{i=s+1}^{s+t}(M_i \cup 1+M_i)) \oplus (\oplus_{j>s+t} R_j)  \\
\qquad\qquad\qquad\qquad\qquad\qquad\qquad\qquad  \textrm{if $k \equiv 2$ {\rm (mod $6$)},} \\
(\oplus_{i=1}^{s} M_i)  \oplus (\oplus_{i=s+1}^{s+t}(R_i \setminus (M_i \cup 1+M_i))) \oplus (\oplus_{j>s+t} R_j)  \\
\qquad\qquad\qquad\qquad\qquad\qquad\qquad\qquad  \textrm{if $k \equiv 3$ {\rm (mod $6$)},} \\
(\oplus_{i=1}^{s} 2+M_i)  \oplus (\oplus_{i=s+1}^{s+t}(M_i \cup 1+M_i)) \oplus (\oplus_{j>s+t} R_j)  \\ 
\qquad\qquad\qquad\qquad\qquad\qquad\qquad\qquad \textrm{if $k \equiv 4 $ {\rm (mod $6$)},} \\
(\oplus_{i=1}^{s} 1+M_i)  \oplus (\oplus_{i=s+1}^{s+t}(R_i \setminus (M_i \cup 1+M_i))) \oplus (\oplus_{j>s+t} R_j)  \\
\qquad\qquad\qquad\qquad\qquad\qquad\qquad\qquad  \textrm{if $k \equiv 5$ {\rm (mod $6$)}.}
\end{cases}
\end{align*}
\end{theorem}

Using Theorem \ref{thm:eu-sum}, we directly obtain 
a necessary and sufficient condition under which every element of $R$ is a sum of its exceptional units. 

\begin{corollary}  \label{cor:eu-sum}
Assume that each $q_i \ge 3$, $i=1,\ldots, n$.  
Then, every element of $R$ is a sum of its exceptional units if and only if 
there is at most one $q_i$ equal to $3$ and at most one $q_j$ equal to $4$ 
for $i,j=1, \ldots, n$.  
\end{corollary}

\subsection{Multiplicative structure of exceptional units}
\label{sec:eu2}

Here, we want to determine the multiplicative structure of exceptional units of $R$.

For a unit $u \in R^*$, define the set 
\begin{equation*}
\Theta_{k,R}(u)=\{(x_{1},\ldots,x_{k})\in (R^{**})^{k} :x_{1}x_2 \cdots x_{k} = u\},
\end{equation*} 
and denote 
\begin{equation*}\label{theta}
\theta_{k,R}(u)=|\Theta_{k,R}(u)|.
\end{equation*} 
That is, $\theta_{k,R}(u)$ is the number of ways to represent $u$ as a product of $k$ exceptional units. 

By Corollary \ref{cor:exist1}, if there is some $q_i=2$, then $R^{**}=\emptyset$. 
Certainly we need to exclude this case. 
Applying the same arguments  as in \cite{Miguel}, we obtain an exact formula for $\theta_{k,R}(u)$. 

\begin{theorem} \label{thm:theta} 
Assume that each $q_i \ge 3$, $i=1, \ldots, n$. 
Then, for  any integer $k\ge 2$  and any $u = (u_1, \ldots, u_n) \in R^*$, we have 
$$
\theta_{k,R}(u) = \prod_{i=1}^{n} m_i^{k-1} (q_i-1)^{-1} \sigma_{k,R_i}(u_i), 
$$
where 
$$
\sigma_{k,R_i}(u_i)= \begin{cases}
(q_i-2)^{k} + (-1)^{k}(q_i-2) &  \textrm{if $u_i \in 1+ M_i$,} \\
(q_i-2)^{k} + (-1)^{k+1} & \textrm{if $u_i \not\in 1+ M_i$.}  \\
\end{cases}
$$
\end{theorem}

Then, we can determine the multiplicative structure of $R^{**}$.

\begin{theorem} \label{thm:eu-prod}
 The following hold: 

{\rm (i)} If each $q_i>3$, $i=1,\ldots, n$, then we have 
$$
R^{**} \cdot  R^{**} = R^*.
$$

{\rm (ii)} 
If $q_1=\cdots =q_s=3$ $(s\ge 1)$ and $q_j > 3$ for each $j>s$, then for any integer $k\ge 2$ we have  
$$
\prod_{i=1}^{k} R^{**} = 
\begin{cases}
(\oplus_{i=1}^{s} 1+M_i)  \oplus (\oplus_{j>s} R_j^*) & \textrm{if $k$ is even,} \\ 
(\oplus_{i=1}^{s} R_i^* \setminus 1+M_i)  \oplus (\oplus_{j>s} R_j^*) & \textrm{otherwise.}
\end{cases}
$$
In particular, every unit of $R$ is a product of its exceptional units  if and only if $s=1$. 
\end{theorem}

Finally, we determine under which condition the ring $R$ can be generated by its exceptional units 
(that is, every element of $R$ can be represented as either a sum of exceptional units or a sum of products of exceptional units). 

\begin{corollary}  \label{cor:exunit2}
Assume that each $q_i \ge  3$, $i=1, \ldots, n$. 
Then, the ring $R$ is generated by its exceptional units  if and only if there is at most one $q_i$ equal to $3$. 
\end{corollary}

\subsection{Applications} 
\label{sec:app}

So far, there have been extensive studies on additive unit representations in global fields; 
see \cite{BFT} for a survey. 
Here, we apply our main results to the case of exceptional units in Dedekind domains.  






Let $D$ be a Dedekind domain. 
Notice that every exceptional unit of $D$ automatically yields an exceptional unit in $D/I$ for any ideal $I$ of $D$.  
By Corollary \ref{cor:exist1} (or noticing the simple fact that the binary field has no exceptional unit), we immediately have:  

\begin{corollary}  \label{cor:exist2}
If $D$ has a prime ideal of norm $2$, then $D^{**}=\emptyset$. 
\end{corollary}

In view of quadratic number fields, one can see that the condition in Corollary \ref{cor:exist2} is sufficient but not necessary. 
In fact, by definition it is easy to see that the only exceptional units in quadratic fields 
are the roots of the four polynomials:  
$$
X^2 -X+1, \quad X^2 -3X+1, \quad X^2 -X-1, \quad X^2 +X-1. 
$$
They correspond to the quadratic fields $\Q(\sqrt{-3})$, $\Q(\sqrt{5})$. 
For each of the two fields, the ring of integers is generated by its exceptional units.

Using Corollaries \ref{cor:eu-sum} and  \ref{cor:exist2}, we directly obtain:  

\begin{corollary}  \label{cor:exunit}
Not every element of $D$ is a sum of its exceptional units if one of the following conditions holds:
\begin{itemize}
\item $D$ has a prime ideal of norm $2$; 

\item $D$ has at least two prime ideals of norm $3$; 

\item $D$ has at least two prime ideals of norm $4$. 
\end{itemize}
\end{corollary}

The following is a direct consequence of Corollary \ref{cor:exunit2}. 

\begin{corollary}  \label{cor:exist3}
$D$ cannot be generated by its exceptional units if $D$ has at least two prime ideals of norm $3$. 
\end{corollary}

We remark that Corollary \ref{cor:exist3} can be proved by the following simple argument without using 
Corollary \ref{cor:exunit2}: if $P$ and $Q$ are two prime ideals of norm $3$, 
then every element of $D$, which is congruent to $0$ modulo $P$ and to $1$ modulo $Q$, 
cannot be a sum of exceptional units or a sum of products of exceptional units, 
because the only exceptional unit modulo $P$  is $2$.  

The condition in Corollary \ref{cor:exist3} is also only sufficient. 

\begin{example}
Choose $D$ to be the ring of integers of the quadratic field $\Q(\sqrt{21})$. 
Then, the prime $2$ is inert and the prime $3$ is  ramified in $D$. 
By Corollary \ref{cor:exunit2}, for any non-zero ideal $I$ the quotient ring  $D/I$ is generated by its exceptional units. 
However, $D$ has no exceptional unit. 
\end{example}

\section{Proofs}
\label{Prof} 

\subsection{Proof of Corollary \ref{cor:u-sum}}

(i)
Notice that for any $c=(c_1,\ldots,c_n) \in R$, if $q_i > 2$, then 
$$
\mu_{2,R_i}(c_i) >0. 
$$ 
This together with Theorem \ref{thm:psi} implies the result in (i). 

(ii)
Now, assume that $q_i=2$, then by Theorem \ref{thm:psi}, we have 
\begin{equation} \label{eq:mu}
\mu_{k,R_i}(c_i)= \begin{cases}
1+(-1)^{k}&  \textrm{if $c_i \in M_i$,} \\
1+(-1)^{k+1} & \textrm{if $c_i \in R_i \setminus M_i$.}
\end{cases}
\end{equation}
Then, letting $k$ be even and applying \eqref{eq:mu}, we obtain 
$$
\sum_{i=1}^{k} R^* =  (\oplus_{i=1}^{s} M_i ) \oplus  (\oplus_{j>s} R_j).
$$

Similarly, letting $k$ be odd and using \eqref{eq:mu} we obtain the second identity.

\subsection{Proof of Theorem \ref{thm:eu-sum}}

(i)
Notice that given $c=(c_1,\ldots,c_n) \in R$, for each $i=1, \ldots, n$, 
$$
\rho_{2,R_i}(c_i) > 0 \quad \textrm{if $q_i > 4$.}
$$ 
This together with Theorem \ref{thm:phi} implies the result in (i). 

(ii)
If $q_i=3$, we have 
$$
\rho_{2,R_i}(c_i) = 3\left(\sum_{\substack{j=0 \\ j \equiv c_i \,\textrm{\rm (mod $M_i$)}}}^{2}\binom{2}{j} -1\right) , \quad \textrm{for any $c_i\in R_i$}.
$$
As $q_i=3$, the residue classes modulo $M_i$ can be represented by $0,1,2$ respectively. 
So, $\rho_{2,R_i}(c_i) \ne 0$ if and only if $c_i \in 1+M_i $. 
Then, using Theorem \ref{thm:phi} we get 
\begin{equation}  \label{eq:32}
R^{**} + R^{**} = (\oplus_{i=1}^{s} 1+M_i )  \oplus (\oplus_{j>s} R_j). 
\end{equation}
We again have 
$$
\rho_{3,R_i}(c_i) = 3\left(\sum_{\substack{j=0 \\ j \equiv c_i \,\textrm{\rm (mod $M_i$)}}}^{3}\binom{3}{j} -3\right) , \quad \textrm{for any $c_i \in R_i$}.
$$
Thus, $\rho_{3,R_i}(c_i) \ne 0$ if and only if $c_i \in M_i$. 
So,  by Theorem \ref{thm:phi} we obtain  
\begin{equation} \label{eq:33}
R^{**} + R^{**} + R^{**} = (\oplus_{i=1}^{s} M_i)  \oplus (\oplus_{j>s} R_j).
\end{equation}
Then, combining \eqref{eq:32} with \eqref{eq:33} we get the identities in (ii).

(iii) 
If $q_i = 4$, for even $k$ we have 
$$
\rho_{k,R_i}(c_i) = 4\sum_{\substack{j=0 \\ j \equiv c_i \,\textrm{\rm (mod $M_i$)}}}^{k}\binom{k}{j} , \quad \textrm{for any $c_i \in R_i$}. 
$$
So, $\rho_{k,R_i}(c_i) \ne 0$ if and only if $c_i\in M_i$ or $c_i \in 1+M_i$, where one should note that the characteristic of the residue field $R_i / M_i$ is $2$ 
(because $q_i=4$). 
Thus, by using Theorem \ref{thm:phi} we get the first identity in (iii).

For the second identity, we note that  for odd $k$, we have 
$$
\rho_{k,R_i}(c_i) = 4\left(\sum_{\substack{j=0 \\ j \equiv c_i \,\textrm{\rm (mod $M_i$)}}}^{k}\binom{k}{j} - 2^{k-1}\right) , \quad \textrm{for any $c_i \in R_i$}.
$$
Since the characteristic of the residue field $R_i / M_i$ is $2$, if $c_i \in M_i$, then $j \equiv c_i \equiv 0 \,\textrm{\rm (mod $M_i$)}$ for any even integer $j$. 
Also, if $c_i  \in 1+M_i$, then $j \equiv c_i \equiv 1\,\textrm{\rm (mod $M_i$)}$ for any odd integer $j$. 
Thus, $\rho_{k,R_i}(c_i) \ne 0$ if and only if $c_i \not\in M_i$ and $c_i \not\in 1+ M_i$. 
So, similarly we obtain the second identity. 

(iv)
The desired results in (iv) follow directly from (ii) and (iii).

\subsection{Proof of Theorem \ref{thm:theta}}

Applying the same arguments as in \cite{Miguel}, we have the following two lemmas, whose proofs we omit. 

\begin{lemma}\label{le1}
For any $u=(u_1,\ldots, u_n) \in R^*$, we have 
$$
\theta_{k,R}(u) = \prod_{i=1}^{n} \theta_{k,R_i}(u_i). 
$$
\end{lemma}

\begin{lemma}\label{le2}
For each $i=1, \ldots, n$ and for any unit $u_i \in R_i^*$ we have 
 $$
 \theta_{k,R_i}(u_i) = m_i^{k-1} \theta_{k,R_i/M_i}(u_i).
 $$
 \end{lemma}

 Let $\F_q$ be a finite field of $q$ elements. 
Recall that a multiplicative character $\chi$ of $\F_q^*$ is a homomorphism from $\F_q^*$ to the complex roots of unity. 
The trivial character $\chi_0$ is the one sending every element of $\F_q^*$ to $1$. 
Let $G_q$ be the group of multiplictive characters of $\F_q^*$, and let $G_q^* = G_q \setminus \{\chi_0\}$. 
Then, $|G_q| = q-1$. 
Furthermore, we have the following orthogonality relations (for instance, see \cite{Lidl}): 
\begin{equation*} 
\sum_{a\in \F_q^*}\chi(a)=\begin{cases}
q-1 & \textrm{if $\chi=\chi_0$,}   \\
0&  \textrm{otherwise;}
\end{cases}
\end{equation*}
and 
\begin{equation*} 
\sum_{\chi \in G_q}\chi(a)=\begin{cases}
q-1 & \textrm{if $a=1$,}   \\
0&  \textrm{otherwise.}
\end{cases}
\end{equation*}

\begin{proof}[Proof of Theorem \ref{thm:theta}] 
Based on the above lemmas, we only need to calculate $\theta_{k,\F_q}(c)$ for $c \in \F_q^*$ and $q >2$.  
Using the above formulas about multiplicative characters, we obtain 
\begin{align*}
\theta_{k,\F_q}(c)
&=|\{(x_{1},\ldots,x_{k})\in \big(\F_q^{**}\big)^{k} :x_{1}x_2 \cdots x_{k} = c \}| \\
&=\sum_{x_{1}\in \F_q^{**}}\sum_{x_{2}\in \F_q^{**}}\cdots\sum_{x_{k}\in \F_q^{**}}
\frac{1}{q-1}\sum_{\chi \in G_q}\chi(x_{1}\cdots x_{k}/c)\\
&=\frac{1}{q-1}\sum_{\chi \in G_q}\Big(\sum_{x_{1}\in \F_q^{**}}\chi(x_{1})\Big)\cdots\Big(\sum_{x_{k}\in \F_q^{**}}\chi(x_{k})\Big)\chi(c^{-1})\\
&=\frac{1}{q-1}\Big(\sum_{\chi \in G_q^{*}}\Big(\sum_{x_{1}\in \F_q^{**}}\chi(x_{1})\Big)\cdots\Big(\sum_{x_{k}\in \F_q^{**}}\chi(x_{k})\Big)\chi(c^{-1}) +(q-2)^{k}\Big)\\
& \qquad (\textrm{since}~\F_q^{**}=\F_q^*\setminus\{1\}).
\end{align*}
Notice that for any $\chi \in G_q^{*}$, we have 
$$
0=\sum_{x_{i}\in \F_q^*}\chi(x_{i})=1+ \sum_{x_{i}\in \F_q^{**}}\chi(x_{i}).
$$
Then, we further have 
\begin{align*}
\theta_{k,\F_q}(c)
&=\frac{1}{q-1}\Big((q-2)^{k} + (-1)^k\sum_{\chi \in G_q^{*}}\chi(c^{-1})\Big)\\
&= \begin{cases} 
\frac{1}{q-1}\left((q-2)^{k} + (-1)^{k}(q-2)\right)&  \textrm{if $c = 1$,} \\
\frac{1}{q-1}\left((q-2)^{k} + (-1)^{k+1}\right)& \textrm{if $c \ne 1$.} 
\end{cases}
\end{align*}
This, together with Lemma \ref{le1} and Lemma \ref{le2}, implies the desired result. 
\end{proof}

\subsection{Proof of Theorem \ref{thm:eu-prod}}

(i)
Notice that given $u=(u_1,\ldots,u_n) \in R^*$, for each $i=1, \ldots, n$, we have 
$$
\sigma_{2,R_i}(u_i) > 0 \quad \textrm{if $q_i > 3$.}
$$ 
This together with Theorem \ref{thm:theta} implies the result in (i). 

(ii)
If $q_i=3$, we have 
\begin{equation}   \label{eq:sigma}
\sigma_{k,R_i}(u_i)= \begin{cases}
1+(-1)^{k} &  \textrm{if $u_i \in 1+ M_i$,} \\
1+(-1)^{k+1}& \textrm{if $u_i \not\in 1+ M_i$.}  
\end{cases}
\end{equation}
Then, letting $k$ be even and applying \eqref{eq:sigma}, we get 
$$
\prod_{i=1}^{k}R^{**} = (\oplus_{i=1}^{s} 1+M_i)  \oplus (\oplus_{j>s} R_j^*). 
$$
This completes the proof of the first identity. 

Similarly, letting $k$ be odd and using \eqref{eq:sigma} we obtain the second identity.

\subsection{Proof of Corollary \ref{cor:exunit2}}

The sufficient part follows directly from Corollary \ref{cor:u-sum} (i) and Theorem \ref{thm:eu-prod}. 

For the necessary part, we suppose that 
 $q_1 = q_2 = 3$. 
By assumption, the ring $R$ is generated by its exceptional units. 
Then, the ring $R_1 \oplus R_2$ is also generated by its exceptional units, and so is the ring $R_1/M_1 \oplus R_2/M_2$. 
On the other hand, since both finite fields $R_1/M_1$ and $R_2/M_2$ have only three elements, 
 we in fact have 
$$
(R_1/M_1 \oplus R_2/M_2)^{**}=\{(2,2)\}, 
$$ 
which generates the subset 
$$
\{(0,0),(1,1),(2,2)\}. 
$$
So, the ring $R_1/M_1 \oplus R_2/M_2$ cannot be generated by its unique exceptional unit. 
This leads to a contradiction.

\section*{Acknowledgment} 
The authors would like to thank the referees for careful reading and valuable comments. 
This work was supported by the National Natural Science Foundation  of China, Grant No. 11501212. 
The research of Min Sha was supported by the Macquarie University Research Fellowship.

\bibliography{central}

\end{document}